\title{Diameter of a direct power of a finite group}
\author{Nasim Karimi \footnote{Universidade Federal de Alagoas, Campus A. C. Simões - Av. Lourival Melo Mota, s/n, Cidade Universit\'aria, Macei\'o, Alagoas, 57072-900, Brasil
	Tel:  +55 (82) 3214-1418, 
Fax: 	+55 (82) 3214-1418 }\\
nakareme@gmail.com
}
\date{\today}
  
\documentclass[11pt]{article}
\usepackage{hyperref}
\usepackage{amsmath,amsthm,amsfonts,amssymb,amscd}
\usepackage[svgnames]{xcolor}
\usepackage{mathrsfs}
\usepackage{eepic}
\usepackage{amsmath}
\usepackage{imakeidx}
\newtheorem{theorem}{\bf Theorem}[section]

\newtheorem{lem}[theorem]{\bf Lemma}
\newtheorem{defi}[theorem]{\bf Definition}
\newtheorem{rem}[theorem]{\bf Remark}
\newtheorem{cor}[theorem]{\bf Corollary}
\newtheorem{ex}[theorem]{\bf Example}

\newtheorem{conj}[theorem]{\bf Conjecture}
\newtheorem{prop}[theorem]{\bf Proposition}
\def\rank{\ensuremath{\mathrm{rank}}}

\def\diam{\ensuremath{\mathrm{diam}}}
\def\Cay{\ensuremath{\mathrm{Cay}}}
\def\ord{\ensuremath{\mathrm{ord}}}

\newtheorem{notation}[theorem]{\bf Notation}

\begin{document}
\maketitle
\begin{abstract}
We present two conjectures concerning the diameter of a direct power of a finite group. The first conjecture states that 
the diameter of $G^n$ with respect to any generating set is at most $n(|G|-\rank(G))$; and  
 the second one states that there exists a generating set $A$, of minimum size, for $G^n$ such that
the diameter of $G^n$ with respect to $A$ is at most  $n(|G|-\rank(G))$.
We will establish evidence for each of the above mentioned conjectures.
 \end{abstract}

\section{Introduction}
There are several results concerning the calculation of the diameter of a finite group in the literature (e.g.,\cite{Babai:2006,Babai&Seress:1992,Seress&Helfgott:2014}). But, it is the first time that the author is explicitly interested in the diameter of a direct power of a finite group. In fact, the motivation comes from a semigrooup problem. More  precisely, in \cite{Karimi:2015}  the notion of depth parameters for a finite semigroup has been introduced  and  beside estimating the depth parameters of some families of finite semigroups, the behaviour of one of the depth parameters with respect to the direct product and wreath product have been verified . There, the diameter of a direct power of a group play an essential role in the control of the length of elements in the group of units of a semigroup. 
  
 Let $G$ be a finite group with a generating set $A$. By the diameter of $G$ with respect to $A$ we mean the maximum over $g \in G$ of the length of the shortest word in $A$ representing $g$. 
 Our definition here is a bit different from the one which  has been usually considered in the literature. Usually group theorists define the diameter to be the maximum over $g \in G$ of the length of a shortest word in $A \cup A^{-1}$ representing $g$. Let us call this version of the diameter to be ``symmetric diameter"\index{symmetric diameter}. An asymptotic estimate of the symmetric diameters of
 non-Abelian simple groups with respect to various types of generating sets can be found in the survey \cite{Seress&Helfgott:2014}, which also lists related work, e.g., on the 
 diameters of permutation groups.  We are interested in the behaviour of the diameter with respect to the direct product. Specially, we focus on  the direct power of a 
 finite group. More precisely, let $G$ be a finite group and $G^n$ be the $n$-th direct power of $G$. Let $A$ be a minimal generating set of $G^n$. Our objective is to 
 find a reasonable answer to the following question. How large can be the diameter of $G^n$ with respect to $A$? A simple argument shows that the diameter of a group with 
 respect to any generating set is bounded above by the group order minus the group rank (see Proposition \ref{upper-bound-diameter-general}). This bound is tight for cyclic groups.  It is obvious that $G^n$ is not cyclic for $n \geq 2$ and $|G| \neq 1$. Then, the following natural question 
 arises. Is there any smaller upper bound (less than $|G^n|-\rank(G^n)$) for the diameter of $G^n$? In fact, $|G|^n-\rank(G^n)$ is exponentially large in terms of $|G|$. 
 The more precise question in which we are really interested is whether the diameter of a direct power of a finite group is polynomially bounded. These questions lead to the following
 conjectures. Throughout this paper, $G^n$ denotes the $n$-th direct power of the group $G$.
 \begin{conj}[strong] 
Let $G$ be a finite group. Then the diameter 
$D(G^n) $ is at most $n(|G|-\rank(G)).$ 
\end{conj}
\begin{conj}[weak] 
Let $G$ be a finite group. Then there exists a generating set $A$ for $G^n$ of minimum size such that
\begin{equation}\label{conjbound}
\diam(G^n,A) \leq n (|G|-\rank(G)).
\end{equation}
\end{conj}
Note that both conjectures hold for trivial groups. 
As the second conjecture is a consequence of the first one, it may be easier to establish. Anyway, each of the proposed conjectures has advantages and disadvantages when attempting to prove them. The difficulty in proving the weak conjecture is dealing with generating sets of minimum size for the direct powers of finite groups. Finding such a generating set is itself a problem. Nevertheless, there exist in the literature many results concerning the computation of the rank of a direct power of a finite group, e.g., \cite{Wiegold:1974,Wiegold:1975,Wiegold:1978,Wiegold:1980,Wiegold&Meier:1981}. On the other hand, every direct power $G^n$ of a finite group has a generating set, called canonical generating set (Definition \ref{canonical-generating-set}), which satisfies  the inequality \eqref{conjbound}. So, the weak conjecture  for groups whose rank is equal to the size of the canonical generating set is true.  For instance, the canonical generating set for direct powers of nilpotent groups is a generating set of minimum size (Corollary \ref{property}). Therefore, nilpotent groups satisfy  the weak conjecture  easily. However, the canonical generating set is not always a generating set of minimum size. Then the difficulty of establishing the weak conjecture  appears when we consider the  groups for which the rank of their direct powers is not equal to the size of the canonical generating set. On the other hand, the strong conjecture concerns arbitrary generating sets. It has the advantage that there are many results in the literature concerning the computation of the diameter of a finite group with respect to an arbitrary generating set. Hence, we may use the upper bounds obtained by other authors to approach the strong conjecture .  

This paper is organised as follows. 
After Preliminaries, in Section \ref{se:abelian}, we show that Abelian groups satisfy the strong conjecture. Section \ref{se:generating} deals with generating sets of 
minimum size for direct powers of finite groups. In Section \ref{se:diameter}, we show that 
the weak conjecture is true for nilpotent groups, symmetric groups and the alternating group $A_4$. Also we show that the weak conjecture holds for dihedral groups under some restrictions on $n$. Finally, in the last section, we present some polynomial upper bounds 
for the diameter of a direct power of a solvable group.

 
\section{Preliminaries}
 Since the notion of diameter of a finite group has different definitions in the literature we introduce here our definitions and notation precisely.
 
Let $G$ be a finite group with  a generating set $A$. Denote by $A^{-1}$ the set $ \{a^{-1} : a \in A\}.$ 
 By the \emph{symmetric length} of an element $g\in G$, with respect to $A$, we mean the minimum length of a sequence which represents $g$ in terms of elements in $A \cup A^{-1}$. Denote this parameter by $l^{s}_A(g)$.

 Now, we have the following different definitions for the diameter of a finite group with respect to a generating set.
 
 \begin{defi}
 Let $G$ be a finite group with generating set $A$. By the \emph{diameter} of $G$ with respect to $A$ we mean  
 $$\diam(G,A):=\max \{l_A(g): g \in G \}.$$ 
\end{defi} 

\begin{defi}
 Let $G$ be a finite group with generating set $A$. By the \emph{symmetric diameter} of $G$ with respect to $A$ we mean  
 $$\diam^s(G,A):=\max \{l^{s}_A(g): g \in G \}.$$ 
\end{defi}  

\begin{notation}
Denote by $D(G)~(\mbox{respectively}~ D^s(G))$  the maximum diameter (respectively symmetric diameter) over all generating sets of $G$.
\end{notation} 

\begin{rem}
Let $G$ be a finite group with a generating set $A$.  For $g \in G$ we have $l_{A}^s(g) \leq l_{A}(g)$ and $D^s(G) \leq D(G)$.
\end{rem}

The terminology of diameter and symmetric diameter of a group comes from the diameter of the Cayley graph and the directed Cayley graph of a group.   
 
 \begin{defi}\label{cayly-graph}
By the \emph{Cayley graph} of a group $G$ with respect to a generating set $A$ we mean the graph whose set of vertices is $G$ and such that there is an edge between $g_1,g_2 \in G$ if and only if $g_1^{-1}g_2 \in A \cup A^{-1}.$  We denote this graph by $\Cay(G,A)$. 
 \end{defi}

 \begin{defi}\label{directed-cayly-graph}
By \emph{directed Cayley graph} of group $G$ with respect to a generating set $A$ we mean the directed graph whose set of vertices is $G$ and such that there is an edge from $g_1$ to $g_2$ if and only if $g_1^{-1}g_2 \in A.$  We denote this graph by $\overrightarrow{\Cay}(G,A)$.
 \end{defi} 
 \begin{defi}
  A directed graph is called \emph{strongly connected} if it contains a directed path from $u$ to $v$ and a directed path from $v$ to $u$ for every pair of vertices $u, v$. 
 \end{defi}
 The distance between two vertices in a connected graph is the length of the shortest path between them. In the case of a strongly connected directed graph the distance between two vertices $u$ and $v$ is defined as the length of the shortest path from $u$ to $v$. Notice that, in contrast with the case of undirected graphs, the distance between $u$ and $v$ dose not  necessarily coincide with the distance between $v$ and $u$, so it is not a distance in the metric sense of the word. 
 We consider the \emph{diameter} of a (strongly connected directed) graph as usual, that is, the maximum of distances  between two vertices over all pairs of vertices in the vertex set. It is easy to see that the symmetric diameter of a group with respect to a generating set is equal to the diameter of the corresponding Cayley graph; and the diameter of a group with respect to a generating set is equal to the diameter of the corresponding directed Cayley graph (note that the directed Cayley graph of a group is always strongly connected). 
 
 The following proposition gives a general upper bound for the diameter of a finite group.
 \begin{prop}\label{upper-bound-diameter-general}
 Let $G$ be a finite group. The inequality $D(G) \leq |G|- \rank(G)$ holds.
 \end{prop}
 \begin{proof}
 Let $X$ be an arbitrary generating set of $G$. It is enough to show that $\diam(G,X) \leq |G|-\rank(G)$. We can suppose, without loss of generality, that $1 \not \in X$. Let 
 $\diam(G,X)=t$. There exist $g \in G$ and $x_1,x_2\ldots,x_t \in X$ such that $g=x_1x_2\cdots x_t$ and $t$ is the smallest number for which $g$ has such a kind of decomposition in $X$. Hence $x_1,x_1x_2,\ldots,x_1x_2\cdots x_t$ are $t$ distinct non identity elements of $G$. On the other hand, $X$ has $|X|-1$ elements distinct from $x_1,x_1x_2,\ldots,x_1x_2\cdots x_t$. By adding the identity to these distinct elements we get $|G| \geq t + |X|-1 +1$, which gives the inequality $\diam(G,X) \leq |G|-|X|$. Now, the result follows from $\rank(G) \leq |X|$.     
 \end{proof} 
\section{Abelian groups and the strong conjecture }\label{se:abelian}

A canonical decomposition of a finite Abelian group $G$ is an expression of $G$ as a direct product of cyclic subgroups whose orders $m_1,m_2,\ldots,m_k$ satisfy $m_i \mid m_{i-1}$ for $i=1,2,\ldots,k$. Then, $m_1,m_2,\ldots,m_k$ are the invariants  of $G$.
Following the notation in \cite{Klopsch&Lev:2003}, we say $G$ is of type $(m_1,m_2,\ldots,m_k)$. 
By using the following theorem which has been proved in \cite{Klopsch&Lev:2009} it is easy to show that Abelian groups satisfy the strong conjecture.

\begin{theorem}\cite{Klopsch&Lev:2009} \label{Abelian-diameter}
 Let $G$ be a finite Abelian group of type $(m_1,m_2,\ldots,m_k)$ . We have 
$$D(G) = \sum_{i=1}^k (m_i-1).$$
\end{theorem}

\begin{cor}\label{cor-diameter-abelian}
Let $G^n$ be a direct power of a finite Abelian group $G$. The following inequality holds:
$$ D(G^n) = n D(G) \leq n \, (|G|-\rank(G)).$$
\end{cor}
\begin{proof}
Let $G=C_1 \times C_2 \times \cdots \times C_k$ be a canonical decomposition of $G$. We have 
$$ G^n\cong C_1^n \times C_2^n \times \cdots \times C_k^n.$$ Using Theorem \ref{Abelian-diameter}, we obtain  
$$D(G^n)= \sum_{i=1}^k n\, (|C_i|-1) = n \, \sum_{i=1}^k (|C_i|-1) \leq n \, (|G|-\rank(G)).\qedhere$$
\end{proof}

\section{Generating sets of minimum size}\label{se:generating} 
 Finding a generating set of minimum size for a direct power of a group is itself a problem. Nevertheless, there exist in the 
literature many results concerning the computation of the rank of a direct power of a finite group, e.g., \cite{Wiegold:1974,Wiegold:1975,Wiegold:1978,Wiegold:1980,Wiegold:1987}. We use such kind of results to find generating sets of 
minimum size for direct powers of some families of finite groups such as the symmetric group $S_n$ and the dihedral group $D_n$. With a different approach (see \cite{Hall:1936}) we establish generating sets 
of minimum size for the direct powers $A_5^2, A_5^3, A_5^4$, where $A_5$ is the alternating group of degree five. 
\begin{defi}\label{canonical-generating-set} 
Let $G$ be a finite group with a generating set $A$. By the \emph{canonical generating set} of $G^n$ with respect to $A$, we mean the set   
$$C^n(A):=\{(1,\ldots,\overbrace{a}^{i\,\mathrm{th}},\ldots,1) : i \in \{1,2,\ldots, n \}, a \in A \}.$$
\end{defi}
\begin{lem}\label{lem:canonical}
Let $G$ be a finite group with a generating set $A$. For $n\geq 1$ the following inequality holds:
$$\diam(G^n,C^n(A)) \leq n ~\diam(G,A).$$
\end{lem}
\begin{proof}
For given $(g_1,g_2,\ldots,g_n) \in G^n$ we have $(g_1,g_2,\ldots,g_n) = \prod_{i=1}^n(1,\ldots,g_i,\ldots,1).$  Then,  we have $$l_{C^n(A)}(g_1,g_2,\ldots,g_n) \leq \sum_{i=1}^n l_{C^n(A)}(1,\ldots,g_i,\ldots,1).$$ By the definition of $C^n(A)$,  for $i \geq1$,  we have $$l_{C^n(A)}(1,\ldots,g_i,\ldots,1) \leq \diam(G,A)$$ which gives the desired conclusion.
\end{proof}
\begin{lem}\label{rank-inequality} {\rm \cite{Wiegold:1974}}
Let $G$ be a finite group and $k$ be a positive integer. The following inequalities hold:
\begin{equation}
k\, \rank (G/G') \leq \rank (G^k) \leq k\, \rank(G),
\end{equation}
where $G'$ is the commutator subgroup of $G$.
\end{lem}
The following is an application of Lemma \ref{rank-inequality}.
\begin{cor}\label{property}
Let $G$ be a finite group. If $\rank(G)=\rank(G/G')$, then the following equality holds:
\begin{equation}\label{group-property}
\rank(G^n)=n\, \rank(G).
\end{equation}
  In particular, nilpotent groups satisfy this property.
\end{cor}
\begin{proof}
The first statement is an immediate consequence of Lemma \ref{rank-inequality}. We prove the second statement. Note that, if $H$ is a homomorphic image of a finite group $G$, then $\rank(H) \leq \rank (G)$. Therefore, it is enough to show that $\rank(G) \leq \rank(G/G')$ for every finite nilpotent group $G$. Let $A= \{g_1G',g_2G',\ldots,g_kG' \}$ be a generating set of $G/G'$ of minimum size. Consider an arbitrary element $g \in G$. There exist some $i_1,i_2,\ldots, i_l \in \{1,2,\ldots,k\}$ such that $gG'= g_{i_1}g_{i_2}\ldots g_{i_l}G'$. This shows that $G$ is generated by $\{g_1,g_2,\ldots ,g_k\}$ together with some elements in $G'$. Because $G$ is nilpotent, it is generated by $\{g_1,g_2,\ldots ,g_k\}$ alone, see \cite[page 350]{Magnus&Karrass&Solitar:1976}. Therefore, we get $\rank(G) \leq \rank(G/G')$, which completes the proof.
\end{proof}
\begin{defi}
A group is said to be \emph{perfect} if it equals its own commutator subgroup.
\end{defi}
\begin{lem}\label{rank-G^n}
Let $G$ be a finite group which is not perfect. If $G$ can be generated by $k$ elements of mutually coprime orders, then 
$$\rank(G^n)=n,$$ for $n \geq k$.
\end{lem}
\begin{proof}
Because $G$ is not perfect, it follows from Lemma \ref{rank-inequality} that $\rank(G^n) \geq n$. Suppose
$A=\{a_1,a_2,\ldots,a_k\}$ is a generating set of $G$ such that the $a_i$'s are of mutually coprime orders. Let $n \geq k$. We construct a generating set of size $n$ for $G^n$. For $1\leq i \leq n$, define the elements $g_i \in G^n$ as follows:
\begin{align*}  
&g_i=(1,\ldots, \overbrace{a_1}^{i\,\mathrm{th}},a_2,\ldots,a_k,\ldots,1) ~ \mbox{for}~ 1 \leq i \leq n-k+1,\\ 
&g_i=(a_{n-i+2},a_{n-i+3},\ldots,a_k,1,\ldots, 1,\overbrace{a_1}^{i\,\mathrm{th}},\ldots,a_{n-i+1})~ \mbox{for}~ n-k+2 \leq i \leq n.
\end{align*}
We prove that $C=\{g_1,g_2,\ldots,g_n\}$ is a generating set of $G^n$. If we show that $C$ generates $C^n(A)$, then we are done. Choose an arbitrary element $(1,\ldots,a_i,\ldots,1) \in C^n(A)$. Since the $a_i$'s are of mutually coprime orders, there exists a positive integer $\ell$ such that 
\begin{equation*}(1,\ldots,a_i,\ldots,1) =
(1,\ldots, a_1,\ldots,a_i,\ldots,a_k,\ldots,1)^{\ell}.
\end{equation*}
This yields the desired conclusion.
\end{proof}

\subsection{Symmetric groups $S_n$}
Denote by $\ord(g)$ the order of a group element $g$.
\begin{lem}\label{Sn-generating-set}
 For $n \geq 3$ the symmetric group ${S}_n$ can be generated by two elements of coprime order. 
 \end{lem}
 \begin{proof}
  Define the permutations  $a,a'$ and $b$ as follows:
$$ a=  \begin{pmatrix}
1 & 2 & 3 & \ldots & n\\
2 & 3 & 4 & \ldots & 1
\end{pmatrix}, 
 a'=  \begin{pmatrix}
1 & 2 & 3 & \ldots & n\\
1 & 3 & 4 & \ldots & 2
\end{pmatrix}
 $$ and
$$b=  \begin{pmatrix}
1 & 2 & 3 & \ldots & n\\
2 & 1 & 3 & \ldots & n
\end{pmatrix}.$$ 
It is known that the full cycle $a$ and the transposition $b$ generate $S_n$ \cite{Hungerford:1980}. On the other hand, note that
$a'b = a$. Hence the sets $\{a,b\}$ and $\{a',b\}$ are generating sets of $S_n$. Note that 
$$ \ord(a)=n,~  \ord(b)=2,~  \ord(a')=n-1.$$  Therefore, for odd $n$, the set $A=\{a,b\}$ and, for even $n$, the set $A'=\{a',b\}$ are the desired generating sets.
\end{proof}
\begin{cor}\label{rank-Snk}
For $k \geq 2$, the equality
$$ \rank({S}_n^k) =k,$$ holds.
\end{cor}
\begin{proof}
Since the derived subgroup of $S_n$ is $A_n$ and $S_n$ is not perfect,  the assertion follows immediately by Lemmas \ref{rank-G^n} and \ref{Sn-generating-set}. 
\end{proof}
\subsection{ Dihedral groups $D_n$}
\begin{defi}
The dihedral group $D_n$ is the group of symmetries of a regular polygon with $n$ sides.
\end{defi}
We consider the \emph{dihedral group} $D_n$ as a subgroup of $S_n$.
\begin{prop}\label{rank-Dnk}
 For odd $n$ and $k \geq 2$,
the rank of $D_n^k$ is $k$ 
and, for even $n$, the rank of 
 $D_n^k$ is $2k$.
\end{prop}
\begin{proof}
Suppose for the moment that $n$ is odd. Let
$$ 
a=  \begin{pmatrix}
1 & 2 & 3 & \ldots & n\\
2 & 3 & 4 & \ldots & 1
\end{pmatrix},\ 
b=  \begin{pmatrix}
1 & 2 & 3 & \ldots& n-1 & n\\
1 & n & n-1 & \ldots &3 & 2
\end{pmatrix}
 .$$ It is easy to check that $A=\{a,b\}$ is a generating set of $D_n$. Because $a,b$ have  coprime orders, Lemma \ref{rank-G^n} gives the desired conclusion.

Now let $n$ be even. The commutator subgroup of  $D_n$ is a cyclic group of order $\frac{n}{2}$, and the quotient group is the Klein four-group and thus 
$$\rank(\frac{D_n}{D'_n})=2.$$
On the other hand, we have $\rank(D_n)=2.$
Using Corollary \ref{property}, we get $\rank(D_n^k)=2k$. 
\end{proof}
\subsection{Alternating groups $A_n$}
The following example gives a generating set of minimum size for a direct power of the alternating group $A_4$. 
\begin{ex} \label{alternating-A4}
It is easy to see that $A_4$ is generated by the two elements $$ \alpha=(1 ~2)(3~4) , ~\beta =(1~2~3).$$ Since $A_4$ is not perfect and $\alpha , \beta $ have coprime orders by Lemma \ref{rank-G^n}, the rank of $A_4^n$ is equal to $n$, for $n \geq 2$.
\end{ex}
Recall that the alternating groups $A_n$ for $n \geq 5$ are simple. Since non-Abelian simple groups are perfect, the alternating groups $A_n$ for $n \geq 5$ are perfect. There is a different approach to compute the rank of the direct power of perfect groups using the Eulerian function  of a  group (see \cite{Hall:1936,Wiegold:1974}). The following lemma is a consequence of the results in \cite{Hall:1936}.

\begin{lem}\label{Hall}
Let $G$ be a non-Abelian simple group. If $G$ is generated by $n$ elements, then 
the set  $\{(a_{i1},a_{i2}\ldots,a_{ik}):i=1,\ldots,n\}$ will generate $G^k$ if and only if the following conditions are satisfied:
\begin{enumerate}
\item the set $\{a_{1i},a_{2i},\ldots,a_{ni}\}$ is a generating set of  $G$ for $i=1,\ldots,k$;
\item there is no automorphism $f: G\rightarrow G$ which maps $(a_{1i},a_{2i},\ldots,a_{ni})$ to $(a_{1j},a_{2j},\ldots,a_{nj})$ for any $i\neq j$.
\end{enumerate}
\end{lem}
Furthermore, in \cite{Hall:1936} Hall shows that the alternating group $A_5$ satisfies Lemma \ref{Hall} with $n=2$ for $1 \leq k \leq 19$ and  not for $k \geq 20$.

Therefore, the following is an immediate consequence of Lemma \ref{Hall}.
\begin{cor}\label{Hall2}
A pair $(s_1,\ldots,s_k), (t_1,\ldots,t_k)$ will generate $A_5^k$ if and only if the following conditions are satisfied:
\begin{enumerate}
\item the set $\{s_i,t_i\}$ is a generating set of  $A_5$ for $i=1,...,k$;
\item there is no automorphism $f: A_5\rightarrow A_5$ which maps $(s_i,t_i)$ to $(s_j,t_j)$ for any $i\neq j$.
\end{enumerate}
Furthermore, $k=19$ is the largest number for which these conditions can be satisfied.
That is,
the rank of $A_5^k$ is equal to $2$ if and only if $1 \leq k \leq 19$.
\end{cor}
\subsection{Solvable Groups}
The following theorem has been proved by Wiegold in \cite{Wiegold:1975}.
\begin{theorem}\label{rank-solvable-group}\cite{Wiegold:1975}
Let G be a finite non-trivial solvable group, and set $\rank(G)=\alpha,~\rank(G/G')=\beta.$ The equality $\rank(G^n)=\beta n$ holds, for $n \geq \alpha / \beta$. 
\end{theorem}


\section{Diameter of direct powers of groups}\label{se:diameter}

Here, we present some families of groups which satisfy the weak conjecture. 
\begin{rem}\label{trivial}
Every group $G$ with the property
\begin{equation}\label{group-rank}
\rank(G^n)=n\,\rank(G),
\end{equation} 
satisfies the weak conjecture. More precisely, if $A$ is a generating set of $G$ with minimum size then $C^n(A)$ is a generating set of minimum size for $G^n$. So, the statement is obvious by Lemma \ref{lem:canonical}.

Then it suffices to justify the weak conjecture for groups that do not have property \eqref{group-rank}. In particular, by Corollary \ref {property}, every nilpotent group has property \eqref{group-rank} .
\end{rem}
\begin{prop}
 Let $G$ be a solvable group such that $\alpha=\rank(G)$ and $\beta=\rank(G/G')$. Then $G^n$ satisfies the weak conjecture for $n \geq \frac{\alpha}{\beta}$.
\end{prop}
\begin{proof}
By Theorem \ref{rank-solvable-group}, we have $\rank(G^n)=n\rank(G/G')$. Since $G/G'$ is Abelian, we have $\rank((G/G')^n)=n \rank(G/G')$. Moreover, since $G^n/(G')^n \cong (G/G')^n$ and $(G')^n \cong (G^n)'$ we get $\rank(G^n)=\rank(G^n/(G^n)')$. It means that $G^n$ satisfies the property of Corollary \ref{property}. Now the result follows from Remark \ref{trivial}. 
\end{proof}

\begin{rem}\label{simple}
Let $(g_1,g_2,\ldots,g_n) \in G^n=\langle A \rangle $. Since
$ (g_1,g_2,\ldots,g_n)$ is a product of $n$ elements of the form  $(1,\ldots,g_i,\ldots,1),$ then we have 
\begin{equation}\label{length1} 
l_A(g_1,g_2,\ldots,g_n)\leq \sum_{i=1}^n l_A(1,\ldots,g_i,\ldots,1). 
\end{equation}
\end{rem}

The following easy lemma gives an upper bound for the diameter of a direct power of a finite group $G$ in terms of the diameter of the group $G$. We use this lemma in the next section to prove that the symmetric group $S_n$ satisfies the weak conjecture. 

\begin{lem} \label{general-diameter-direct-power}
 For a given generating set $A$ of $G^n$,  we have
$$\diam(G^n,A) \leq Ml_A(C^n(X)) ~\sum_{i=1} ^{n} \diam(G,A \pi_i) ,$$
where $$X=\bigcup_{i=1}^{n} ( A \pi_i \setminus\{1\}),$$
and $\pi_i: G^n \rightarrow G $ maps each element to its $i$-th cordinate. 
\end{lem}
\begin{proof} Let $(g_1,g_2,\ldots,g_n) \in G^n$. 
Since, for $i=1,2,\ldots ,n, ~A\pi_i \setminus \{1\}$ is a generating set of $G$, there exist
$$g_{i1} , g_{i2},\ldots,g_{ik_i} \in A\pi_i \setminus \{1\}, ~\mbox{for some}~ k_i \leq \diam(G,A\pi_i),$$ such that
$$g_i=g_{i1} g_{i2}\ldots g_{ik_i}.$$
This gives
$$(1,\ldots,g_i,\ldots,1)=\prod_{j=1}^{k_i}(1,\ldots,g_{ij},\ldots,1),$$
hence,
\begin{equation}\label{length2}
l_A(1,\ldots,g_i,\ldots,1) \leq \sum_{j=1}^{k_i}l_A(1,\ldots,g_{ij},\ldots,1).
\end{equation}
Substituting \eqref{length2} into \eqref{length1}, we get
\begin{align*}
l_A(g_1,g_2,\ldots,g_n) & \leq \sum_{i=1}^n \sum_{j=1}^{k_i} l_A(1,\ldots,g_{ij},\ldots,1)\\
& \leq \sum_{i=1}^n \sum_{j=1}^{k_i} Ml_A(C^n(X))\\
& =  Ml_A(C^n(X))\sum_{i=1}^n k_i \\
& \leq Ml_A(C^n(X))\sum_{i=1}^n \diam(G,A \pi_i),\\
\end{align*}
in which the second inequality is due to the fact that $$(1,\ldots,g_{ij},\ldots,1) \in C^n(X).$$ 
This finishes the proof.
\end{proof}
\subsection{Symmetric groups $S_n$ and the weak conjecture}

Here, our goal is to show that the symmetric group $S_n$ satisfies the weak conjecture. First, we apply Lemma \ref{general-diameter-direct-power} to show that $S_n$ satisfies the weak conjecture for $n \geq 7$. Then, we discuss the case  $n \leq 6$.

\begin{ex} \label{1}
Let $A,A'$ be the generating sets defined in Lemma \ref{Sn-generating-set} and $C,C'$ be the corresponding  generating sets of ${S}_n^k$ constructed in the proof of Lemma \ref{rank-G^n}, for odd and even $n$, respectively.
Note that 
$$ X=\bigcup_{i=1}^{k} ( C \pi_i \setminus\{1\})=\{a,b\},~ X'=\bigcup_{i=1}^{k} ( C' \pi_i \setminus\{1\})=\{a',b\}.$$
Therefore, we have
\begin{align*}
C^k(X)=\{(1,\ldots,a,\ldots,1) : 1 \leq i \leq k \} \cup \{(1,\ldots,b,\ldots,1) : 1 \leq i \leq k\}\\
C^k(X')=\{(1,\ldots,a',\ldots,1) : 1 \leq i \leq k\} \cup \{(1,\ldots,b,\ldots,1) : 1 \leq i \leq k\}.
\end{align*}
If $n$ is odd, for $i =1,2,\ldots, k$, we have 
\begin{align*}
(1,\ldots,a,\ldots,1)=(1,\ldots, a,b,\ldots,1)^{n+1},\\
(1,\ldots,b,\ldots,1)=(1,\ldots, a,b,\ldots,1)^n.
\end{align*}
If $n$ is even, for $i =1,2,\ldots, k$, we have
 \begin{align*}
(1,\ldots,a',\ldots,1)=(1,\ldots, a',b,\ldots,1)^n,\\
(1,\ldots,b,\ldots,1)=(1,\ldots,a' ,b,\ldots,1)^{n-1}.
\end{align*}
It follows that
\begin{align}\label{length-odd}
 Ml_C(C^k(X)) \leq n+1, \mbox{if} ~n~ \mbox{is odd},
\end{align}

\begin{align}\label{length-even}
 Ml_{C'}(C^k(X')) \leq n, \mbox{if} ~n~ \mbox{is even}.
\end{align}
\end{ex}

\begin{lem} \label{diameter-Sn}
Let $A,A'$ be the generating sets defined in Lemma \ref{Sn-generating-set}. Then the following inequalities hold:
$$\diam({S}_n,A) \leq (n-1)(2n-3)(n+1),$$ 
$$\diam({S}_n,A') \leq (n-1)(2n-3)(2n+1) .$$
\end{lem}
\begin{proof}
A simple calculation shows that, for $1 \leq i \leq n-1,$ 
$$
 a^{n-i+1}ba^{i-1}=(a'b)^{n-i+1}b(a'b)^{i-1}=(i,i+1).
$$
Therefore, we have 
$$
l_A(i,i+1) \leq n+1, ~~
l_{A'}(i,i+1) \leq 2n+1.
$$
Let $(i,i+k)$ be an arbitrary transposition in ${S}_n$. Since
$$\begin{array}{lll}
(i,i+k)&=&(i,i+1)(i+1, i+2)
 \cdots(i+k-1, i+k)(i+k-2, i+k-1)\\
&&\cdots(i+1,i+2)(i,i+1),
\end{array}$$
 every transposition is a product of at most $2n-3$ transpositions of the form $(i,i+1)$. It follows that 
$$l_A(i,i+k) \leq (2n-3)(n+1),~~l_{A'}(i,i+k)\leq (2n-3)(2n+1).$$
Consider a permutation $\sigma$ in ${S}_n$. Because every permutation in ${S}_n$ is a product of at most $n-1$ transpositions, we have  
$$Ml_A(\sigma) \leq (n-1)(2n-3)(n+1) ,~~ Ml_{A'}(\sigma)=(n-1)(2n-3)(2n+1) .$$
The proof is complete.
\end{proof}
\begin{lem}\label{diameter-Snk}
Let $A$ and $A'$ be the generating sets defined in Lemma \ref{Sn-generating-set} and  $C$ and $C'$ be the corresponding generating sets of ${S}_n^k$ constructed in the proof of Lemma \ref{rank-G^n}, for odd and even $n$, respectively.
For $n \geq 3$ and $k \geq 2$, we have 
$$ \diam({S}_n^k,C) \leq  k (n-1)(2n-3)(n+1)^2,$$ provided that $n$ is odd, and we have  
$$ \diam({S}_n^k,C') \leq k n(n-1)(2n-3)(2n+1),$$ provided that $n$ is even.
\end{lem}
\begin{proof}
Using Lemmas \ref{general-diameter-direct-power}, \ref{diameter-Sn} and the inequalities \eqref{length-odd},\eqref{length-even} in Example \ref{1}, we have
\begin{align*}
\diam({S}_n^k,C) &\leq Ml_C(C^k(A)) ~\sum_{i=1} ^{k} \diam({S}_n,A) \\
& \leq k~ Ml_C(C^k(A)) \, \diam({S}_n,A)\\
& \leq k~Ml_C(C^k(A)) (n-1)(2n-3)(n+1)\\
& \leq k ~(n+1)(n-1)(2n-3)(n+1),
\end{align*}
provided that $n$ is odd. 
Similar arguments apply to the case where $n$ is even, which yields the second inequality.
\end{proof}
\begin{cor}
 The symmetric group $S_n$ satisfies the weak conjecture for $n \geq7$. 
\end{cor}
\begin{proof}
 We have $(n-1)(2n-3)(n+1)^2 \leq 2 (n^2-1)^2\leq 2 n^4$. We  induct on $n\geq 7$  to show that  $2n^4 \leq n!-2$. The inequality holds for $n=7$. For $n >7$, since
 $2(n+1)^4=2n^4+8n^3+12n^2+8n+2 \leq 10n^4$, then the induction hypothesis gives the required conclusion.  
 Whence, for $n\geq7$, we have
$$(n-1)(2n-3)(n+1)^2 \leq n!-2.$$
 
 Also we know that $n(n-1)(2n-3)(2n+1) \leq 4(n^2-1)^2\leq 4n^4$. By induction on $n\geq 8$ we show that $4n^4 \leq n!-2$. The inequality holds for $n=8$. For $n >8$, since
 $4(n+1)^4=4n^4+16n^3+24n^2+16n+4 \leq 20n^4$, then the induction hypothesis gives the required conclusion.  
Whence, for $n\geq8$, we have
$$n(n-1)(2n-3)(2n+1) \leq n!-2.$$ 

Now the result is immediate by Lemma \ref{diameter-Snk}. 
\end{proof}
 Note that the symmetric group $S_2$ is Abelian so satisfies the weak conjecture. We show that the weak conjecture is true for the symmetric group $S_n$, for $n=4,5,6$. 
Let $A,A'$ be the generating sets defined in Lemma \ref{Sn-generating-set} and $C,C'$ be the corresponding generating sets constructed in the proof of Lemma \ref{rank-G^n}. We can calculate the diameter of ${S}_n$ with respect to $A,A'$ for small values of $n$ by using a package called GRAPE \footnote{You may find this package at http://www.gap-system.org/Packages/grape.html.}in GAP. Here is the result:  
\begin{align*}
 \diam(S_4,A')&=7,\\
\diam(S_5,A)&=11,\\
\diam(S_6,A')&=17.
\end{align*}
Therefore, by Lemma \ref{general-diameter-direct-power} and the inequalities \eqref {length-odd}, \eqref {length-even} we have
\begin{align*}
\diam(S_4^k,C')& \leq  28 k\\
\diam(S_5^k,C) &\leq  66k\\
\diam(S_6^k,C') &\leq  102k.
\end{align*}
It follows that the symmetric group ${S}_n$  satisfies the weak conjecture for $n=5,6$ but the above upper bound for $S_4$ is greater than the upper bound of  the weak conjecture. We perform an alternative computation to establish the weak conjecture for $S_4$. By Remark \ref{simple}, it suffices to show that for $1 \leq i \leq k$, the elements
$(1,\ldots,g_i,\ldots,1)$ may be presented as products of at most $22$ generators in the generating set $C'$. 
Since we have
 \begin{align} 
(1,\ldots,a',b,\ldots,1)^2=(1,\ldots,1,b^2,\ldots,1),\\
(1,\ldots,a',b,\ldots,1)^4=(1,\ldots,1,b,\ldots,1),\\
(1,\ldots,a',b,\ldots,1)^3=(1,\ldots,a',1,\ldots,1),
\end{align}
then 
\begin{align} \label{basic-elements1}
 l_{C'} (1,\ldots,b^2,\ldots,1) \leq 2,\\
\label{basic-elements2}
 l_{C'} (1,\ldots,b,\ldots,1) \leq 4 ,\\
\label{basic-elements3}
l_{C'} (1,\ldots,a',\ldots,1) \leq 3,
\end{align}
for $1 \leq i \leq n.$
On the other hand, the elements of $S_4$ in the generating set $\{a', b \}$ can be represented as follows:
\begin{align*}
S_4=&\{a', b, a'^2,a' b, b a', b^2, a' b a', a' b^2,b a' b, b^2 a', (a' b)^2,a' b^2 a', (b a')^2, \\
& b a' b^2, b^2 a' b, (a' b )^2 a',(a' b )^2 b, a' b^2 a' b, b a' b^2 a',b^2 a' b a', (a' b )^2 b a',\\
& a' b^2 a' b a',b a' b^2 a' b, (a' b )^2 b a'  \}.
\end{align*} 
Now, it is easy to check that for every $g \in S_4$ and for  $1 \leq i \leq k$ the elements
$(1,\ldots,g,\ldots,1)$ can be written as a product of at most $19$ generators in the generating set $C'$ as we required.
 
\subsection{Upper bound for the diameter of direct power of dihedral groups}
\begin{prop}\label{diameter-Dnk}
For $n \geq 3$ and $k \geq 1$, there exists a generating set $C$ of minimum size for $D_n^k$ such that 
\begin{equation}
\diam(D_n^k,C) \leq \left \{ \begin{array}{llll}
k (2n-2)& \mbox{if} ~n~ \mbox{is even}, \\
\frac{n+1}{2}+(k-1)(2n-1) & \mbox{if} ~n~ \mbox{is odd}.
\end{array} \right .
\end{equation} 
\end{prop}
\begin{proof}
According to Proposition \ref{rank-Dnk} and Remark \ref{trivial}, it is enough to consider the case where $n$ is odd. Let $A=\{a,b\}$ be the generating set defined in Proposition \ref{rank-Dnk} and $C$ be the associated generating set of $D_n^k$ constructed in the proof of Lemma \ref{rank-G^n}.
We prove that  
$$ \diam(D_n^k, C) \leq \frac{n+1}{2}+(k-1)(2n-1) .$$ 
Note that every rotation in $D_n$ is a power of $a$ and every reflection in $D_n$ is a power of $a$ multiplied by $b$.
It follows that every element in $D_n$ can be written in the form $a^rb^s$ for some $0 \leq r \leq n-1$ and $s \in \{0,1\}$.
Choose an arbitrary element 
$(x_1,x_2,\ldots,x_k) \in D_n^k.$ In view of the relations $a^ib=ba^{n-i}$ and $a^i=b a^{n-i}b$, we may write $x_1$ as a word of length $\leq \frac{n+1}{2}$ on $A$. Then there exist $y_2,y_3,\ldots,y_k \in D_n$ such that $$(x_1,x_2,\ldots,x_k)=(x_1,y_2,\ldots,y_k)(1,y_2^{-1}x_2,y_3^{-1}x_3,\ldots,y_k^{-1}x_k),$$ and $$l_C (x_1,y_2,\ldots,y_k)\leq \frac{n+1}{2}. $$Write $y_i^{-1}x_i=a^{r_i}b^{s_i}$ with $0 \leq r_i \leq n-1$ and $s_i \in \{0,1\}$.  
By Remark \ref{simple}, the proof is completed by showing that for $2 \leq i \leq k $, 
$$Ml_{C}(1,\ldots,a^{r_i}b^{s_i},\ldots,1) \leq 2n-1.$$
We will do this by considering the following four cases.
The case where $s_i=0$, $r_i$ is even:
 \begin{align*}
&(1,\ldots,a^{r_i}b^{s_i},\ldots,1)\\
&=(1,\ldots,a^{r_i},\ldots,1)=(1,\ldots,a^{r_i},b^{r_i},\ldots,1)=(1,\ldots,a,b,\ldots,1)^{r_i}.
\end{align*}
Hence, we have
$$l_{C}(1,\ldots,a^{r_i}b^{s_i},\ldots,1) \leq r_i \leq n-1 \leq 2n-1.$$
The case where $s_i=0$,  $r_i$ is odd. We have
 \begin{align*}
&(1,\ldots,a^{r_i}b^{s_i},\ldots,1)\\
&=(1,\ldots,a^{r_i},\ldots,1)
=(1,\ldots,a^{r_i+n},b^{r_i+n},\ldots,1)
=(1,\ldots,a,b,\ldots,1)^{r_i+n}.
\end{align*}
Hence, we have
$$l_{C}(1,\ldots,a^{r_i}b^{s_i},\ldots,1) \leq r_i+n \leq 2n-1.$$
The case where $s_i=1$,  $r_i$ is even:
 \begin{align*}
&(1,\ldots,a^{r_i}b,\ldots,1)\\
=&(1,\ldots,a^{r_i},\ldots,1)(1,\ldots,\overbrace{b}^{i\,\mathrm{th}},\ldots,1)\\
=&(1,\ldots,a^{r_i},b^{r_i},\ldots,1) (1,\ldots,a^n,\overbrace{b^n}^{i\,\mathrm{th}},\ldots,1)\\
=&(1,\ldots,a,b,\ldots,1)^{r_i} (1,\ldots,a,\overbrace{b}^{i\,\mathrm{th}},\ldots,1)^n.
\end{align*}
Hence, we have
$$l_{C}(1,\ldots,a^{r_i}b^{s_i},\ldots,1) \leq r_i+n \leq 2n-1.$$
It remains to consider the case where $s_i=1$,  $r_i$ is odd. Note that $$a^{r_i}b=ba^{n-r_i},$$ which entails
 \begin{align*}
&(1,\ldots,a^{r_i}b,\ldots,1)
=(1,\ldots,ba^{n-r_i},\ldots,1)\\
&=(1,\ldots,\overbrace{b}^{i\,\mathrm{th}},\ldots,1)(1,\ldots,a^{n-r_i},\ldots,1)\\
&=(1,\ldots,a^n,\overbrace{b^n}^{i\,\mathrm{th}},\ldots,1)(1,\ldots,a^{n-r_i},b^{n-r_i},\ldots,1)\\
&=(1,\ldots,a,\overbrace{b}^{i\,\mathrm{th}},\ldots,1)^n(1,\ldots,a,\overbrace{b}^{{i+1}-th},\ldots,1)^{n-r_i}.
\end{align*}
Hence, we have
$$l_{C}(1,\ldots,a^{r_i}b^{s_i},\ldots,1) \leq 2n-r_i \leq 2n-1.$$
The proof is complete.
\end{proof}
 Now the following corollary is immediate by Proposition \ref{diameter-Dnk}.
\begin{cor}
The weak conjecture holds for $D_n^k$ if $n$ is even or $k \leq \frac{3(n-1)}{2}$.
\end{cor}

\subsection{Alternating groups and the weak conjecture}

\begin{prop}
The alternating group $A_4$ satisfies the weak conjecture.
\end{prop}
\begin{proof}
As we mentioned before in Example \ref{alternating-A4}, the generating set $C$ constructed in Lemma \ref{rank-G^n} is a generating set of minimum size for $A_4^n$ for $n \geq 2$. We show that $\diam(A_4^n,C) \leq 10n$. Let $(g_1,g_2,\ldots,g_n) \in A_4^n$. By Remark \ref{simple}, it is enough to show that  $l_C(1,\ldots,1, g_i,1,\ldots,1) \leq 10$, for $1 \leq i \leq n$.
Because of the following equalities 
\begin{align*}
(1,\ldots,\overbrace{\alpha}^{i\,\mathrm{th}},\beta,\ldots,1)^3=(1,\ldots,\overbrace{\alpha}^{i\,\mathrm{th}},1,\ldots,1),\\
(1,\ldots,\alpha,\overbrace{\beta}^{i\,\mathrm{th}},\ldots,1)^4=(1,\ldots,1,\overbrace{\beta}^{i\,\mathrm{th}},\ldots,1),\\
(1,\ldots,\alpha,\overbrace{\beta}^{i\,\mathrm{th}},\ldots,1)^2=(1,\ldots,1,\overbrace{\beta^2}^{i\,\mathrm{th}},\ldots,1),
\end{align*}
we have
\begin{align*}
l_C(1,\ldots,\overbrace{\alpha}^{i\,\mathrm{th}},\ldots,1) \leq 3,\\
l_C(1,\ldots,\overbrace{\beta}^{i\,\mathrm{th}},\ldots,1) \leq 4,\\
l_C(1,\ldots,\overbrace{\beta^2}^{i\,\mathrm{th}},\ldots,1) \leq 2.
\end{align*}
On the other hand, the elements of $A_4$ can be represented over the generating set $\{\alpha, \beta \}$ as follows:
$$A_4=\{\alpha,\beta,\alpha^2,\alpha \beta,\beta \alpha, \beta^2, \alpha \beta \alpha, \alpha \beta^2, \beta \alpha \beta=\alpha \beta^2 \alpha, \beta^2 \alpha, \beta^2 \alpha \beta, \beta \alpha \beta^2 \}.$$
Now similarly to the proof of Proposition \ref{diameter-Dnk} the length of $(1,\ldots,\overbrace{g}^{i\,\mathrm{th}},\ldots,1)$ in the generating set $C$ is at most $10$ for every element $g \in A_4$, which completes the proof.
\end{proof}

\begin{defi}
By an \emph{$n$-basis} of a group $G$ we mean any ordered set of $n$ elements $x_1,x_2,\ldots,x_n$ of $G$ which generates $G$. Furthermore, two $n$-bases $x_1,x_2,\ldots,x_n$ and  $y_1,y_2,\ldots,y_n$ of $G$ will be called \emph{equivalent} if there exists an automorphism $\theta$ of $G$ which transforms one into the other:
$$ x_i \theta = y_i,$$
for each $i=1,2,\ldots,n$. Otherwise the two bases will be called \emph{non-equivalent}. 
\end{defi}

\begin{ex}
We show that the weak conjecture is true for $A_5^k$ for $k=2,3,4$.
\end{ex}
\begin{proof}
Let $a=(1~2)(3~4) , b=(1~2~3~4~5) $. It is easy to see that the pairs 
\begin{align*}
( a, b),(b,a),( a, b^2 ),(b^2,a)
\end{align*}
 are four non-equivalent 2-basis of $A_5$. Using Corollary \ref{Hall2} we build generating sets of size two for $A_5^k$, $k=2,3,4$. The result is as follows. Let
 \begin{align*}
&C_1=\{a,b\},\\
&C_2=\{(a,a), (b,b^2)\},\\
&C_3=\{(a,b,a), (b,a,b^2)\},\\
&C_4=\{(a,b,a,b^2), (b,a,b^2,a)\}.
\end{align*} Then the sets $C_1,\ C_2,\ C_3$ and $C_4$ are generating sets of minimum size for the groups $A_5,\ A_5^2,\ A_5^3$ and $A_5^4$, respectively. 
Using GAP we check that $\diam(A_5,C_1)=10$ and $\diam(A_5^2,C_2)=18$. Let $(x,y,z)$ be an arbitrary element in $A_5^3$. Since $(x,z) \in A_5^2=\langle (a, a), (b , b^2) \rangle,$  there exists a word $w$ representing $(x,z)$ over the generating set $C_2$ of length at most $18$. Hence, $(x,y,z)=(w_1,\bar{w},w_2)(1,\bar{w}^{-1} y, 1)$, where $w_1,w_2$ are the first and second components of $w=(x,z)$, respectively. The word $\bar{w}$ is a word in the alphabet $a,b$ corresponding to the word $w$, in which the pairs $(a,a), (b,b^2)$ are substituted by $b, a$, respectively. It follows that
$$ l_{C_3}(x,y,z) \leq l_{C_3}(w_1,\bar{w},w_2)+l_{C_3}(1,\bar{w}^{-1} y, 1).$$ It is clear that $$ l_{C_3}(w_1,\bar{w},w_2) \leq 18.$$ On the other hand,  we have $ l_{C_3}(1,\bar{w}^{-1} y, 1) \leq 60$,
since $$ \diam(A_5,C_1) =10,~(1,b,1)=(a,b,a)^6,~ (1,a,1)=(b,a,b^2)^5 .$$ Therefore, $ l_{C_3}(x,y,z) \leq 18 +60=78$, which implies $\diam(A_5^3,C_3) \leq 78$.   

Let $(x,y,z,w) \in A_5^4$. Consider the factorization $$(x,y,z,w)=(x,1,z,1)(1,y,1,w)$$ and the equalities
\begin{align*}
&(1,b,1,b^2)=(a,b,a,b^2)^6,~(a,1,a,1)=(a,b,a,b^2)^5,\\
&~(1,a,1,a)=(b,a,b^2,a)^5,~(b,1,b^2,1)=((b,a,b^2,a)^6.
\end{align*}
This leads to the following inequalities, 
\begin{align*}
l_{C_4}(x,y,z,w) &\leq l_{C_4}(x,1,z,1)+l_{C_4}(1,y,1,w)\\
& \leq \diam(A_5^2,C_2)  Ml_{C_4} \{(a,1,a,1),(b,1,b^2,1)\} \\
&+ \diam(A_5^2,C_2)Ml_{C_4} \{(1,a,1,a),(1,b,1,b^2)\}\\
& \leq 18 \times6+ 18 \times 6= 216,
\end{align*}
which gives $\diam(A_5^4,C_4) \leq 216$.
\end{proof}

\subsection{Upper bounds for the diameter of a direct power of a solvable group}
Despite our attempts to establish the strong conjecture and the weak conjecture for solvable groups, we could not do it yet. In this section we will present two upper bounds for the diameter of $G^n$, where $G$ is a solvable group. Although these upper bounds do not coincide with the proposed upper bound in the conjectures, they grow polynomially with respect to $n$. Since solvable groups have a derived series of finite length our strategy is to find a relation between the diameter of a solvable group and the diameter of its derived subgroup. For this we need to establish a relation between the generating sets of the group and the generating sets of its subgroups. The following lemma, well known as \emph{Schreier Lemma}, gives a generating set for a subgroup of a group with respect to a generating set of the whole group. The generators of the subgroup are usually called Schreier generators. Using Schreier generators we derive a relation between the diameter of a group and the diameter of its subgroup.

\begin{defi}
Let $H$ be a subgroup of a group $G$. By a right transversal for $G$ mod $H$, we mean a subset of $G$ which intersects every right coset $Hg$ in exactly one element.
\end{defi}

\begin{rem}
Let $G$ be a finite group with a generating set $X$ and a normal subgroup $H$. It is easy to see that the set $HX=\{Hx :~ x \in X\}$ is a generating set of $G/H$. Given an arbitrary element $Hg \in G/H$, $Hg$ can be written as a product of at most $D(G/H)$ elements in $HX$. Hence, there exist $x_1,x_2,\ldots,x_{D(G/H)} \in X$ such that $Hg=Hx_1 Hx_2 H \ldots Hx_{D(G/H)} =Hx_1x_2\ldots x_{D(G/H)} $. It shows that there always exists a right transversal $T$ for $G$ mod $H$ such that $$Ml_{X}(T) \leq D(G/H),~ 1 \in T.$$
\end{rem}

\begin{lem} {\rm \cite{Seress:2003}}\label{Schreier}
Let $H \leq G=\langle X \rangle $ and let $T$ be a right transversal for $G$ mod $H$, with $1 \in T.$ Then the set 
$$ \{ t x t_1^{-1} \mid t,t_1  \in T , x \in X, t x t_1^{-1} \in H \}$$ generates $H$.   
\end{lem}

Using Schreier's Lemma leads to the following observations which we are going to apply for establishing the main result.
The first one is \cite[Lemma 5.1]{Babai&Seress:1992}.
\begin{lem}\label{break-diameter}
If $1 \not = N $, $N \triangleleft G$, then the following inequalities hold:
$$ D^s(G) \leq 2 D^s(G/N) \, D^s(N) + D^s(G/N) +  D^s(N) \leq 4D^s(G/N) \, D^s(N).$$ 
\end{lem}
Here we prove the non symmetric version of Lemma \ref{break-diameter}.

\begin{lem}\label{diameter-Schreier-2}
Let $G$ be a finite group with a generating set $X$ and a normal subgroup $H$. Let $T$ be a right transversal of $G/H$ such that $$ Ml_{X}(T) \leq D(G/H),~ 1 \in T.$$ The following inequality holds:
$$\diam(G,X) \leq D(G/H) + (D(G/H) +1 + Ml_{X}(\{ t^{-1} \mid t \in T \})) D(H).$$ Furthermore, we have
\begin{align*}
D(G^n) \leq D(G^n/H^n)+ (1 + |G| D(G^n /H^n)) D(H^n).
\end{align*}
\end{lem}

\begin{proof}
Given $g \in G$, we have $ g = ht$ for some $h \in H$  and $ t \in T$. Hence, we have  $$ l_{X}(g ) \leq l_{X} (t)+ l_{X}(h).$$ Since $ Ml_{X}(T) \leq D(G/H)$, then $ l_{X}(g ) \leq D(G/H)+ l_{X}(h)$. Using Lemma \ref{Schreier} we get $l_{X}(h) \leq (D(G/H) +1 + Ml_{X}(\{ t^{-1} \mid t \in T \}) ) D(H)$. Combining these two facts gives the upper bound in the first inequality. Now, we prove the second statement. Let $X'$ be a generating set of $G^n$ and let $T'$ be a right transversal of $G^n/H^n$ such that $$ Ml_{X'}(T') \leq D(G^n/H^n).$$ Proceeding as above for the case $n=1$, it suffices to show that $$Ml_{X'}(\{ t^{-1} \mid t \in T'\})\leq (|G|-1) D(G^n /H^n).$$ For given $t \in T'$ we have $$l_{X'}(t) \leq D(G^n/H^n).$$ Since $$ t^{-1}=t^{o(t)-1},$$ then we obtain $$l_{X'}(t^{-1}) \leq (o(t)-1)l_{X'}(t).$$ Hence, we have $$l_{X'}(t^{-1}) \leq (|G|-1)D(G^n/H^n),$$ since the order of any element $g \in G^n$ is at most $|G|$.
The proof is complete.       
\end{proof}

Now we are ready to present two upper bounds for the diameter of a direct power of a solvable group. First, we need to prove the following elementary observation. 

The following corollary is straightforward by using Lemma \ref{diameter-Schreier-2}.
\begin{cor}\label{cor:solvable1}
Let $G$ be a non Abelian solvable group. Let 
$$\{1\}=G^{(l)} \triangleleft G^{(l-1)} \triangleleft \ldots \triangleleft G''\triangleleft  G' \triangleleft G$$ be the derived series of $G$.   
The following inequality holds:
$$D(G^n) \leq n^{l} |G| \prod_{i=0}^{l-2}( |G^{(i)}| +1).$$ 
\end{cor}
 \begin{proof}
For $n=1$ it is obvious.  Let $n\geq 2$. 
Since $(G^k)' = (G')^k$ for $k \geq 1$, then the derived series of $G^n$ is  
\begin{equation}\label{derived series}
\{1\}=(G^{(l)})^n \triangleleft (G^{(l-1)})^n \triangleleft \ldots \triangleleft (G'')^n \triangleleft  (G')^n\triangleleft G^n.
\end{equation}
 Applying Lemma \ref{diameter-Schreier-2} to the group $G^n$ with the subgroup $(G')^n$ gives 
\begin{align*} 
D(G^n) &\leq  D(G^n/ (G')^n) + (1 +|G|D(G^n/ (G') ^n)  D((G')^n)\\
&=D(G^n/ (G')^n)+D((G')^n)+|G|D(G^n/ (G') ^n)  D((G')^n)\\
& \leq D(G^n/ (G')^n)D((G')^n)+|G|D(G^n/ (G') ^n)  D((G')^n)\\
 & = D(G^n/ (G') ^n)  D((G')^n) (1 +|G|), 
\end{align*}
the second inequality follows from the fact that $D(G^n/ (G')^n),D((G')^n) > 1$ and this is because the quotient group $G/G'$ and the commutator subgroup $G'$ are nontrivial and $n\geq 2$.
By repeating  the process for the other subgroups in the series \eqref{derived series} we have 
\begin{equation}\label{diameter-G^n1}
D(G^n) \leq D(G^n/ (G') ^n)D((G')^n/ (G'') ^n) \ldots D((G^{(l-1)})^n) \prod_{i=0}^{l-2}( |G^{(i)}| +1).
\end{equation}
Since for every group $G$ with a normal subgroup $H$ we have $G^n/H^n \cong (G/H)^n$, then 
\begin{equation}\label{diameter-G^n2}
D(G^n) \leq D((G/G') ^n)D((G'/G'') ^n) \ldots D((G^{(l-1)})^n) \prod_{i=0}^{l-2}( |G^{(i)}| +1).
\end{equation}
Since all the quotient groups in the inequality \eqref{diameter-G^n2}  and the group $G^{(l-1)}$ are Abelian by Corollary \ref{cor-diameter-abelian} we get
\begin{align*}
D(G^n)& \leq n^{l} D(G/ G')D(G'/ G'') \cdots\\
& D(G^{(l-2)} / G^{(l-1)}) D((G^{(l-1)}) \prod_{i=0}^{l-2}( |G^{(i)}| +1)\\
& \leq n^{l} |G/ G'| |G'/ G''| \cdots |G^{(l-2)} / G^{(l-1)}| |G^{(l-1)}|\prod_{i=0}^{l-2}( |G^{(i)}| +1)\\
& = n^{l} |G|\prod_{i=0}^{l-2}( |G^{(i)}| +1).\qedhere
\end{align*}
\end{proof}

For finding the second upper bound we start by presenting an upper bound for the symmetric diameter of a direct power of a solvable group and then we apply this to find an upper bound for the diameter of such a group. 
 
\begin{prop}\label{solvable-symmetric}
If $G$ is a solvable group then 
$$D^s(G^n) \leq 4^{l-1} n^l\, |G|,$$
where $l$ is the length of the derived series of $G$.
\end{prop}

\begin{proof}
Let  
$$\{1\}=G^{(l)} \triangleleft G^{(l-1)}\triangleleft \cdots \triangleleft G'' \triangleleft G'\triangleleft G$$
be the derived series of the group $G$. Since for $1 \leq i \leq l $ we have $$(G^{(i)})^n = (G^n)^{(i)},$$ the series 
$$\{1\}=(G^{(l)})^n \triangleleft (G^{(l-1)})^n \triangleleft \cdots \triangleleft (G'')^n \triangleleft (G')^n \triangleleft G^n$$
is the derived series of the group $G^n$. Using the second inequality in Lemma \ref{break-diameter}, the maximum of the diameter of the group $G^n$ is bounded above by 
\begin{align}\label{solvable1}
4 ^{l-1} \, D^s(G^n / (G')^n) \, D^s( (G')^n/ (G'')^n) \cdots D^s((G^{(l-2)})^n / (G^{(l-1)})^n) \, D^s((G^{(l-1)})^n).
\end{align}
Whereas, for $0 \leq i \leq l-2$ we have
$$
(G^{(i)})^n / (G^{(i+1)})^n  \cong (G^{(i)}/ G^{(i+1)})^n$$ and the factors in a derived series are Abelian, by Corollary \ref{cor-diameter-abelian} we get 
\begin{equation}\label{solvable2}
D^s(G^{(i)})^n / (G^{(i+1)})^n \leq n \, |G^{(i)}/G^{(i+1)}|= n  \,|G^{(i)}|/|G^{(i+1)}| 
\end{equation}
for $ \,0 \leq i \leq l-2$ and
\begin{equation}\label{solvable3}
D^s((G^{(l-1)})^n ) \leq n \, |G^{(l-1)}|.
\end{equation}
Substituting the inequalities \eqref{solvable2} and \eqref{solvable3} in \eqref{solvable1}, we get 
$$D^s(G^n) \leq 4^{l-1} n^l\, |G|,$$
which is the desired conclusion.
 \end{proof}
We apply the following Lemma to give an upper bound for the diameter by using the symmetric diameter. 
\begin{lem}\label{diameter-symmetric-diameter}
Let $G$ be a finite group and $X$ be a set of generators. The diameter and the symmetric diameter are related as follows:
$$\diam(G,X) \leq 2(\diam^s(G,X)+1)(|X|+1) \ln|G|.$$ 
\end{lem}
\begin{proof}
See \cite[Corollary 2.2]{Babai:2006}.
\end{proof}
\begin{cor}\label{cor:solvable2}
Let $G$ be a solvable group of derived length $l$ and let $A$ be a generating set of $G^n$ of minimum size. Set $\rank(G)=\alpha, ~\rank(G/G')=\beta.$ The following inequality holds, 
 
$$\diam(G^n,A) \leq 2 (4^{l-1} n^l\, |G|+1)(n \beta+1) n \ln |G|,$$
for $n \geq \alpha / \beta.$
In particular, if $G$ is a $p$-group, then 
$$D(G^n) \leq 2 (4^{l-1}n^l \, |G|+1)(n \beta+1) n \ln |G|,$$
for $n \geq 1.$

\end{cor}
\begin{proof}
By Lemma \ref{diameter-symmetric-diameter} we have,
$$\diam(G^n,A) \leq 2 (\diam^s(G^n,A)+1)(|A|+1) n \ln |G|.$$ In addition, $\diam^s(G^n,A) \leq D^s(G^n)$ by definition. Now by using Proposition \ref{solvable-symmetric} and Theorem \ref{rank-solvable-group} we get the desired conclusion. The second statement follows from these two facts: First, if $G$ is a $p$-group then every minimal generating set is a generating set of minimum size, which follows from the Burnside's Basis Theorem \cite{Hall:1976}. Second, by Corollary \ref{property}, if $G$ is a nilpotent group (note that every $p$-group is nilpotent) then $\rank(G)= \rank(G/G')$. 
\end{proof}
  As an example of a non Abelian solvable group which is also a $2$-group we verify the \emph{quaternion group} $Q_8$.
	Let $Q_8=\{\pm 1, \pm i, \pm j, \pm k \}$ be the quaternion group in which $$i^2=j^2=k^2=-1$$ and $$ij=k, jk=i, ki=j, ji=-k, kj=-i,ik=-j.$$
We have $Q_8' \cong Z_2$ and $Q_8 /Q_8' \cong Z_2 \times Z_2$. The length of the derived series of $Q_8$ is $2$. Hence, $l=2$ and $\beta = \rank(Z_2 \times Z_2)=2$ in the notations of corollaries  \ref{cor:solvable1},\ref{cor:solvable2}. Therefore we have   
$$D(Q_8^n) \leq 72n^2 $$ by Corollary \ref{cor:solvable1} and 
$$D(Q_8^n) \leq 2n (32n^2+1)(2n+1)ln (8)$$ by Corollary \ref{cor:solvable2}.

We now present another upper bound for the diameter of the direct power of the quaternion group $Q_8$ in the following example. 
\begin{ex}
For $n\geq1$ we have $D(Q_8^n) \leq 8n^2+3n$.
\end{ex}
\begin{proof} 
Consider the normal subgroup $H=\{1,-1\}$. Let $X$ be a generating set of $Q_8^n$. We have $H^n \triangleleft Q_8^n$. Let $T$ be a right transversal of $Q_8^n$ mod $H^n$ such that
$$1 \in T , Ml_X(T \setminus \{1\}) \leq D(Q_8^n /H^n).$$  
Using Lemma \ref{diameter-Schreier-2} we have
$$\diam(Q_8^n,X) \leq D(Q_8^n/H^n) + (D(Q_8^n/H^n) +1 + Ml_{X}(\{ t^{-1} \mid t \in T \} )) D(H^n).$$ 
On the other hand, since  $H \cong Z_2$ , $Q_8/H \cong Z_2 \times Z_2$, we have
\begin{equation}\label{quaternion}
\diam(Q_8^n,X) \leq 2n + (2n +1 + Ml_{X}(\{ t^{-1} \mid t \in T \})n.
\end{equation}
Since for every $g \in Q_8^n,~ g^4=1$, for every $t \in T,~ t^{-1}= t^3$. Hence, the following inequality holds:
$$l_{X}(t^{-1}) \leq 3 l_{X}(t) \leq 3  D(Q_8^n /H^n) \leq 6n.$$ 
Substituting $Ml_{X}(\{ t^{-1} | t \in T \}$  by $6n$ in \eqref{quaternion} we get
$$D(Q_8^n) \leq 8n^2+3n.\qedhere $$
\end{proof}
\section{Conclusion}
Despite our attempts to prove or disprove the conjectures, they remain open problems and at this point of the work it is very difficult to say something about the validation of them. To prove or disprove the weak conjecture for dihedral groups, alternating groups and solvable groups is still open. Nevertheless, improve the upper bounds in Corollaries \ref{cor:solvable1} and \ref{cor:solvable2} could be one step towards proving the weak or strong conjecture for solvable groups. 

\section{Acknowledgments}
This is part of the author's Ph.D. thesis, written under the supervision of Professors Jorge Almeida and Pedro Silva at the University of Porto with the financial support from FCT (Funda\c{c}\~ao para a Ci\^ncia e a Tecnologia) with the reference SFRH/BD/51170/2010. The author wishes to express her thanks to her supervisors for suggesting the problem and for many stimulating conversations.  

\bibliographystyle{amsplain}
\bibliography{nasimref}

\providecommand{\bysame}{\leavevmode\hbox to3em{\hrulefill}\thinspace}
\providecommand{\MR}{\relax\ifhmode\unskip\space\fi MR }
\providecommand{\MRhref}[2]{%
  \href{http://www.ams.org/mathscinet-getitem?mr=#1}{#2}
}
\providecommand{\href}[2]{#2}
\begin{thebibliography}{10}

\bibitem{Babai:2006}
L.~Babai, \emph{On the diameter of {E}ulerian orientations of graphs},
  Proceedings of the {S}eventeenth {A}nnual {ACM}-{SIAM} {S}ymposium on
  {D}iscrete {A}lgorithms (New York), ACM, 2006, pp.~822--831. \MR{2368881}

\bibitem{Babai&Seress:1992}
L.~Babai and A.~Seress, \emph{On the diameter of permutation groups}, European
  J. Combin. \textbf{13} (1992), no.~4, 231--243. \MR{1179520 (93h:20001)}

\bibitem{Seress&Helfgott:2014}
J.~Bamberg, N.~Gill, T.~P. Hayes, H.~A. Helfgott, A.~Seress, and P.~Spiga,
  \emph{Bounds on the diameter of {C}ayley graphs of the symmetric group}, J.
  Algebraic Combin. \textbf{40} (2014), no.~1, 1--22. \MR{3226815}

\bibitem{Hall:1976}
Jr.~M. Hall, \emph{The theory of groups}, Chelsea Publishing Co., New York,
  1976, Reprinting of the 1968 edition. \MR{0414669 (54 \#2765)}

\bibitem{Hall:1936}
P.~Hall, \emph{The eulerian functions of a group}, Quart. J. Math. (Oxford)
  \textbf{7} (1936), 134--151.

\bibitem{Hungerford:1980}
T.~W. Hungerford, \emph{Algebra}, Graduate Texts in Mathematics, vol.~73,
  Springer-Verlag, New York-Berlin, 1980, Reprint of the 1974 original.
  \MR{600654 (82a:00006)}

\bibitem{Karimi:2015}
N.~Karimi, \emph{reaching the minimum ideal in a finite semigroup}, 2015, PhD
  thesis.

\bibitem{Klopsch&Lev:2003}
B.~Klopsch and V.~F. Lev, \emph{How long does it take to generate a group?}, J.
  Algebra \textbf{261} (2003), no.~1, 145--171. \MR{1967159 (2004e:20093)}

\bibitem{Klopsch&Lev:2009}
\bysame, \emph{Generating abelian groups by addition only}, Forum Math.
  \textbf{21} (2009), no.~1, 23--41. \MR{2494883 (2010c:20068)}

\bibitem{Magnus&Karrass&Solitar:1976}
W.~Magnus, A.~Karrass, and D.~Solitar, \emph{Combinatorial group theory:
  presentations of groups in terms of generators and relations}, Dover
  publications, INC, New York, 1976.

\bibitem{Wiegold&Meier:1981}
D.~Meier and J.~Wiegold, \emph{Growth sequences of finite groups. {V}}, J.
  Austral. Math. Soc. Ser. A \textbf{31} (1981), no.~3, 374--375. \MR{633445
  (82i:20041b)}

\bibitem{Seress:2003}
A.~Seress, \emph{Permutation group algorithms}, Cambridge Tracts in
  Mathematics, vol. 152, Cambridge University Press, Cambridge, 2003.
  \MR{1970241 (2004c:20008)}

\bibitem{Wiegold:1974}
J.~Wiegold, \emph{Growth sequences of finite groups}, J. Austral. Math. Soc. 17
  (1974), 133--141.

\bibitem{Wiegold:1975}
\bysame, \emph{Growth sequences of finite groups. {II}}, J. Austral. Math. Soc.
  \textbf{20} (1975), no.~part 2, 225--229. \MR{0376856 (51 \#13031)}

\bibitem{Wiegold:1978}
\bysame, \emph{Growth sequences of finite groups. {III}}, J. Austral. Math.
  Soc. Ser. A \textbf{25} (1978), no.~2, 142--144. \MR{499355 (80c:20028)}

\bibitem{Wiegold:1980}
\bysame, \emph{Growth sequences of finite groups. {IV}}, J. Austral. Math. Soc.
  Ser. A \textbf{29} (1980), no.~1, 14--16. \MR{566271 (82i:20041a)}

\bibitem{Wiegold:1987}
\bysame, \emph{Growth sequences of finite semigroups}, J. Austral. Math. Soc.
  (Ser.A) 43 (1987), 16--20.

\end{thebibliography}


\providecommand{\bysame}{\leavevmode\hbox to3em{\hrulefill}\thinspace}
\providecommand{\MR}{\relax\ifhmode\unskip\space\fi MR }
\providecommand{\MRhref}[2]{%
  \href{http://www.ams.org/mathscinet-getitem?mr=#1}{#2}
}
\providecommand{\href}[2]{#2}
\begin{thebibliography}{10}

\bibitem{Babai&Abert:2007}
M.~Ab{\'e}rt and L.~Babai, \emph{Finite groups of uniform logarithmic
  diameter}, Israel J. Math. \textbf{158} (2007), 193--203. \MR{2342463
  (2008m:20070)}

\bibitem{Babai:2006}
L.~Babai, \emph{On the diameter of {E}ulerian orientations of graphs},
  Proceedings of the {S}eventeenth {A}nnual {ACM}-{SIAM} {S}ymposium on
  {D}iscrete {A}lgorithms, ACM, New York, 2006, pp.~822--831. \MR{2368881}

\bibitem{Babai&Beals&Seress:2004}
L.~Babai, R.~Beals, and {\'A}.~Seress, \emph{On the diameter of the symmetric
  group: polynomial bounds}, Proceedings of the {F}ifteenth {A}nnual
  {ACM}-{SIAM} {S}ymposium on {D}iscrete {A}lgorithms, ACM, New York, 2004,
  pp.~1108--1112 (electronic). \MR{2291003}

\bibitem{Babai&Hayes:2005}
L.~Babai and P.T. Hayes, \emph{Near-independence of permutations and an almost
  sure polynomial bound on the diameter of the symmetric group}, Proceedings of
  the {S}ixteenth {A}nnual {ACM}-{SIAM} {S}ymposium on {D}iscrete {A}lgorithms,
  ACM, New York, 2005, pp.~1057--1066 (electronic). \MR{2298365}

\bibitem{Babai&Hetyei&Kantor&Lubotsky&Seress:1990}
L.~Babai, G.~Hetyei, W.~M. Kantor, A.~Lubotsky, and {\'A}.~Seress, \emph{On the
  diameter of finite groups}, 31st {A}nnual {S}ymposium on {F}oundations of
  {C}omputer {S}cience, {V}ol.\ {I}, {II} ({S}t.\ {L}ouis, {MO}, 1990), IEEE
  Comput. Soc. Press, Los Alamitos, CA, 1990, pp.~857--865. \MR{1150735}

\bibitem{Bamberg&Gill&Hayes&Helfgott&Seress&Spiga:2014}
J.~Bamberg, N.~Gill, T.~P. Hayes, H.A. Helfgott, {\'A}.~Seress, and P.~Spiga,
  \emph{Bounds on the diameter of {C}ayley graphs of the symmetric group}, J.
  Algebraic Combin. \textbf{40} (2014), no.~1, 1--22. \MR{3226815}

\bibitem{Helfgott&Seress&Zuk:2015}
H.~A. Helfgott, {\'A}.~Seress, and A.j Zuk, \emph{Random generators of the
  symmetric group: diameter, mixing time and spectral gap}, J. Algebra
  \textbf{421} (2015), 349--368. \MR{3272386}

\bibitem{Helfgott&Seress:2014}
H.A. Helfgott and {\'A}.~Seress, \emph{On the diameter of permutation groups},
  Ann. of Math. (2) \textbf{179} (2014), no.~2, 611--658. \MR{3152942}

\bibitem{Hungerford:1980}
T.~W. Hungerford, \emph{Algebra}, Graduate Texts in Mathematics, vol.~73,
  Springer-Verlag, New York-Berlin, 1980, Reprint of the 1974 original.
  \MR{600654 (82a:00006)}

\bibitem{Karimidiameter:2015}
N.~Karimi, \emph{Diameter of a direct power of a finite group},  (2015),
  http://arXiv.org/abs/1506.02695v1 [math.GR].

\bibitem{Karimidepth:2015}
\bysame, \emph{Reaching the minimum ideal in a finite semigroup},  (2015),
  http://arxiv.org/abs/1506.01633v1 [math.GR].

\bibitem{Klopsch&Lev:2009}
B.~Klopsch and V.~F. Lev, \emph{Generating abelian groups by addition only},
  Forum Math. \textbf{21} (2009), no.~1, 23--41. \MR{2494883 (2010c:20068)}

\bibitem{Magnus&Karrass&Solitar:1976}
W.~Magnus, A.~Karrass, and D.~Solitar, \emph{Combinatorial group theory:
  presentations of groups in terms of generators and relations}, Dover
  publications, INC, New York, 1976.

\bibitem{Wiegold&Meier:1981}
D.~Meier and J.~Wiegold, \emph{Growth sequences of finite groups. {V}}, J.
  Austral. Math. Soc. Ser. A \textbf{31} (1981), no.~3, 374--375. \MR{633445
  (82i:20041b)}

\bibitem{GRAPE4.6.1}
L.~H. Soicher, \emph{{GRAPE}, graph algorithms using permutation groups,
  {V}ersion 4.6.1}, http://www.maths.qmul.ac.uk/~leonard/grape/,grape, May
  2012, Refereed GAP package.

\bibitem{Wiegold:1974}
J.~Wiegold, \emph{Growth sequences of finite groups}, J. Austral. Math. Soc. 17
  (1974), 133--141.

\bibitem{Wiegold:1975}
\bysame, \emph{Growth sequences of finite groups. {II}}, J. Austral. Math. Soc.
  \textbf{20} (1975), no.~part 2, 225--229. \MR{0376856 (51 \#13031)}

\bibitem{Wiegold:1978}
\bysame, \emph{Growth sequences of finite groups. {III}}, J. Austral. Math.
  Soc. Ser. A \textbf{25} (1978), no.~2, 142--144. \MR{499355 (80c:20028)}

\bibitem{Wiegold:1980}
\bysame, \emph{Growth sequences of finite groups. {IV}}, J. Austral. Math. Soc.
  Ser. A \textbf{29} (1980), no.~1, 14--16. \MR{566271 (82i:20041a)}

\end{thebibliography}

\end{document}